%% file: main.tex
\documentclass[10pt,conference,english]{ieeeconf}
\IEEEoverridecommandlockouts

\usepackage{nopageno}
\IEEEoverridecommandlockouts
\usepackage{algorithm}
\usepackage{algpseudocode}
\usepackage{ifthen}
\usepackage{graphicx}
\usepackage[usenames,dvipsnames]{color}

\usepackage{amssymb,amsmath,amsthm}
\usepackage{xcolor}
\usepackage{url}
\usepackage{cite}
\usepackage{amsfonts}
\usepackage{graphicx}
\usepackage{color}
\usepackage{subeqnarray}
\usepackage{subcaption}
\allowdisplaybreaks

\newboolean{isfv}
\setboolean{isfv}{true}

\ifthenelse{\boolean{isfv}}{}{\linespread{0.9}}

\newcommand{\lina}[1]{  \ifthenelse{\boolean{showcomments}}
{ \textcolor{green}{(  #1)}} {}  }
\newcommand{\guannan}[1]{\ifthenelse{\boolean{showcomments}}
{ \textcolor{green}{( #1)} } {} }
\newcommand{\guannanrevise}[1]{\ifthenelse{\boolean{showcomments}}
	{ \textcolor{blue}{#1} } {#1} }

\def\ba{\begin{array}}
\def\ea{\end{array}}
\newcommand{\R}{\mathbb{R}}
\newcommand{\btab}{\begin{tabular}}
\newcommand{\etab}{\end{tabular}}

\newtheorem{theorem}{Theorem}

\newtheorem{proposition}[theorem]{Proposition}
\newtheorem{lemma}[theorem]{Lemma}

\newtheorem{definition}{Definition}

\makeatother

\begin{document}

\title{Exploiting Fast Decaying and Locality in Multi-Agent MDP with Tree Dependence Structure}
\author{Guannan Qu, Na Li 
	\thanks{Guannan Qu and Na Li are affiliated with John A. Paulson School of Engineering and Applied Sciences at Harvard University. Email: gqu@g.harvard.edu, nali@seas.harvard.edu. The work was supported by NSF 1608509, NSF CAREER 1553407, AFOSR YIP, and ARPA-E through the NODES program. }}  

\maketitle

\thispagestyle{plain}
\pagestyle{plain}

\begin{abstract}
	This paper considers a multi-agent Markov Decision Process (MDP), where there are $n$ agents and each agent $i$ is associated with a state $s_i$ and action $a_i$ taking values from a finite set. Though the global state space size and action space size are exponential in $n$, we impose local dependence structures and focus on local policies that only depend on local states, and we propose a method that finds nearly optimal local policies in polynomial time (in $n$) when the dependence structure is a one directional tree. The algorithm builds on approximated reward functions which are evaluated using locally truncated Markov process. Further, under some special conditions, we prove that the gap between the approximated reward function and the true reward function is decaying exponentially fast as the length of the truncated Markov process gets longer. The intuition behind this is that under some assumptions, the effect of agent interactions decays exponentially in the distance between agents, which we term ``fast decaying property''. Results in this paper are our preliminary steps towards designing efficient reinforcement learning algorithms with optimality guarantees in large multi-agent MDP problems whose state (action) space size is exponentially large in $n$.

\end{abstract}

\section{introduction}
	Multi-agent Markov Decision Processes (MDP)  have found many applications such as robot swarms, game play, queuing networks, and cyber-physcial systems \cite{application_stone2000multiagent,application_gameplay,application_wu2016optimal,marl_bu2008comprehensive}. In a typical multi-agent MDP, there
	are $n$ agents and each agent $i$ has a state $s_i$ and an action $a_i$, both taking values from finite sets. Further, each agent is associated with stage reward $r_i$ that is a function of $s_i$ (and/or $a_i$), and the total stage reward is the summation of $r_i$. The goal is to find decision policies such that the average total reward is maximized.
	
	A fundamental difficulty in solving multi-agent MDP is that even if individual state and action spaces are small, the entire state profile $(s_1,\ldots,s_n)$ and the action profile $(a_1,\ldots,a_n)$ can take values from a set of size exponentially large in $n$. Such curse of dimensionality renders the problem almost intractable to solve. In fact, it is even difficult to just specify the transition probability matrix of the problem.
	Another challenge is that even if an optimal policy can be found to map a global state $(s_1,\ldots,s_n)$ profile to an action profile $(a_1,\ldots,a_n)$, it is usually impractical to implement such a policy for real-world systems because of the limited information and communication among agents. For example, it may be costly for agent $i$ obtain states other than $s_i$. 
	
	The above challenges motivate us to focus on Multi-agent MDPs with \emph{local dependence structures} and \emph{local policies}. Specifically, we associate the agents with a underlying dependence graph $\mathcal{G}$, and assume that the distribution of $s_i(t+1)$ only depend on the current states of the local neighborhood of $i$ as well as the local $a_i(t)$. Further, we restrict to the class of local polices, where agent $i$ takes action $a_i$ based on its own local state $s_i$. Such structures can be found in many real networks, e.g. epidemic spreading network \cite{epi_mei2017dynamics}, opinion dynamics in social networks \cite{application_chakrabarti2008epidemic,application_llas2003nonequilibrium}, communication \cite{application_communication}, networking\cite{complexity_papadimitriou1999complexity}. 
	
	However, challenges remain. Despite the local dependence structures, these structures can not be directly utilized by most dynamic programming algorithms like value iteration, policy iteration \cite{bertsekas2007dpbook} or reinforcement learning algorithms like $Q$-learning, actor-critic methods \cite{sutton1998introduction,rl_konda2000actor}.
	Though there are approximate dynamic programming methods using function approximation to reduce the computational burden caused by the large state/action space \cite{tsitsiklis1997analysis}, it is not immediately clear how to choose function approximators (or basis vectors in linear function approximators) to achieve both computational efficiency and provable (near-)optimalitiy. This brings the following question of this paper: \emph{can the dependence structure of the MDP be utilized to develop efficient algorithms that provably find a (near-)optimal local policy?}
	
	

 \vspace{3pt}   
\noindent\textbf{Our Contributions.} In this paper, we partially answer this question when the dependence graph is a one-directional tree. We propose a Locality-based Local Policy Search (LLPS) method to search for (near-)optimal local policies where individual agent $i$ make the decision $a_i$ based on the local state $s_i$. The method builds on approximately reward functions which are evaluated using locally truncated Markov processes, where for each agent $i$ we consider a ``truncated'' model consisting of agent $i$ and its $k$-hop local parents, where influence beyond the $k$-hop parent is ignored. For small $k$, these approximated reward functions can be computed and optimized efficiently. 
Further, we show that, under some conditions, the gap between the approximated reward function and the true reward function 
\emph{decays exponentially in $k$}, meaning the resulting local policies are near-optimal with small $k$ (whose computation is efficient). This result is based on what we call the ``fast decaying property'', which means the effect of node $i$'s action on node $j$ decays exponentially fast when the graph distance between $i$ and $j$ increases. We will provide both theoretic and empirical studies that show the fast decaying property holds under some assumptions. 
	
	We would like to clarify that while our current framework is model-based and requires full knowledge of the model, the framework in this paper sets ground work towards designing efficient model-free reinforcement learning algorithms with optimality guarantees in large multi-agent MDP problems whose state (action) space size is exponentially large in $n$.


\subsection{Related Work}\label{subsec:related_work}


\noindent\textbf{Partially-Observable MDPs and Stochastic Games.} By restricting $a_i$ to depend on only the local state $s_i$, the problem for each agent can be viewed as a Partially-Observable MDP (POMDP) \cite[Chapter 5]{bertsekas2005dp_vol1} if every other agent's policy is fixed. By allowing every agent to explore their optimal strategy, our problem shares many similarities to stochastic games \cite{stochasticgame_shapley,stochasticgame_jeffeee}. However, we would like to clarify that we only focus on the cooperative setting where agents cooperatively trying to minimize the total rewards. We also only consider the local policies where $a_i(t)$ is only based on \textit{current} $s_i(t)$. The uncooperative setting and history-based local policies are left for future interest. 


\vspace{4pt}
\noindent\textbf{Computational Complexity Results.} The problem we consider is related to a series of work on the computational complexity of stochastic control where the state space is exponentially large but the problem can be succinctly described. For a overview, see \cite[Section 5.2]{complexity_blondel2000survey}. As pointed out in \cite{complexity_blondel2000survey}, apart from rare examples like the multi-armed bandit MDP problem \cite{complexity_gittins2011multi}, such problems are in general intractable, as demonstrated by restless bandit problems \cite{complexity_whittle1988restless} and queuing network control problems \cite{complexity_papadimitriou1999complexity}. Compared to these work, our model is more structured. 
In the context of \cite{complexity_blondel2000survey}, our contribution can be viewed as that we provide another type of succinctly described MDPs and we identify some conditions under which the problem is tractable. 

\vspace{4pt}
\noindent\textbf{Multi-Agent and Distributed Reinforcement Learning.} Our work is also related to the literature on multi-agent and distributed reinforcement learning which has been actively studied \cite{marl_bu2008comprehensive}. For example, \cite{marl_littman1994markov,marl_claus1998dynamics,marl_littman2001value,marl_nashq} employ a game-theoretic framework to study multi-agent reinforcement learning. Paper \cite{marl_lauer2000algorithm} considers a setting where there is a global state which is accessible to every agent and proposes a variant of $Q$-learning that avoids the use of a exponentially large $Q$ table on the joint action; \cite{marl_peshkin2000learning} considers local policy gradient search by modeling each individual agent as a partial observed MDP. 
A recent line of work \cite{zhang2018fully,kar2013cal,macua2015distributed,mathkar2017distributed,wai2018multi} considers the distributed reinforcement learning problem. However, most of them assume there is a global state $s$ which is accessible to all agents; whereas we consider each agent has an individual state $s_i$ and it is hard to access the global state $(s_1,\ldots,s_n)$.


\vspace{4pt}
\noindent\textbf{Factored MDP.} Another related line of work is factored MDP, see e.g. \cite{factor_guestrin2003efficient}. In a factored MDP, states follow similar local dependence structure as ours, and the reward is assumed to be additive. In \cite{factor_guestrin2003efficient}, a class of linear function approximators are proposed that are similar in spirit to our ``truncation'' technique, and policy-iteration and linear programming methods based on the approximator are proposed. However, the theoretic guarantee in \cite{factor_guestrin2003efficient} is based on a certain projection error, which is not guaranteed to be small, whereas our optimality gap is guaranteed to be small, though we make stronger assumptions compared to \cite{factor_guestrin2003efficient}.

\vspace{4pt}
\noindent\textbf{Epidemic Model.} Our network model share similarities to epidemic networks \cite{epi_cator2012second,epi_sahneh2013generalized,epi_mei2017dynamics}; see cf. \cite{epi_mei2017dynamics} for a comprehensive review. In a typical epidemic model, each state $s_i(t)$ means node $i$ is infected or not, and an infected node can cause neighboring nodes to be infected with a certain probability. 
Despite the similarity in dependence structure, our work use very different analysis tools and have different focuses. 

\vspace{4pt}
\noindent\textbf{Glauber dynamics.} Another related line of work is Glauber dynamics \cite{physics_andrey2015dynamic,mezard2009information} in the statistical physics. In a Glauber Dynamics, there is a underlying graph and state $s_i(t+1)$'s distribution depends on the states of $i$'s local neighborhood. Dynamic belief propagation methods (also known as dynamic message passing, dynamic cavity methods) have been proposed to find the marginalized stationary distribution of such models \cite{physics_andrey2014thesis,physics_neri2009cavity,physics_kanoria2011majority,physics_aurell2012dynamic,physics_altarelli2013optimizing}. However, most of these methods rely on certain ansatz (simplifying assumptions) and there is a lack of theoretic guarantee even when the graph is a tree.

\section{Problem Formulation}\label{sec:problem_formulation}

Consider $n$ agents $\mathcal{N}=\{1,\ldots,n\}$ and each agent is associated with state $s_i\in\{0,1\}$ and action $a_i\in\{0,1\}$.\footnote{In this paper we focus on binary state and action for ease of exposition. } The global state is denoted as $s = (s_1,\ldots,s_n)\in \{0,1\}^n$, and the global action is $a=(a_1,\ldots,a_n)\in \{0,1\}^n$.

\textbf{Dependence Graph.} The $n$ agents are associated with an underlying one-directional tree $\mathcal{G} $ rooted at node $1$. Each node $i$ has $1$-hop parent $\delta_i^1$, $k$-hop parent $\delta_i^k$, and set of children $c_i \subset \mathcal{N}$. We will also use $\Delta_i^k$ to denote the path from $i$ to $i$'s $k$-hop parent, $(i,\delta_i^1,\ldots,\delta_i^k)$, and $\tilde\Delta_i^k$ to denote the path from $i$'th $1$-hop parent to $k$-hop parent, $(\delta_i^1,\ldots,\delta_i^k)$. See Fig.~\ref{fig:truncation} for an illustration. When node $i$ does not have a $k$-hop parent, then $\delta_i^k$ means empty set, $\Delta_i^k$ means the path from $i$ to the root node, and $\tilde\Delta_i^k$ means the path from $i$'s $1$-hop parent to the root node.

\textbf{Transition Probabilities. } We assume that conditioned on the current global state $s(t)$ and global action $a(t)$, the next state $\{s_i(t+1)\}_{i=1}^n $ is generated independently, and the distribution of $s_i(t+1)$ (the next state at $i$) only depends on $s_i(t)$ (the current state state at $i$), $a_i(t)$ (the current action at state $i$), and $s_{\delta_i^1}(t)$ (the current state of the $1$-hop parent of $i$) when $i$ is not the root node. Mathematically, this means the transition probability factorizes in the following manner,
{
\begin{align}
    P(s(t+1)|s(t),a(t)) =   \prod_{i=1}^n P_i(s_i(t+1)| s_i(t), s_{\delta_i^1}(t), a_i(t)). \label{eq:P_factorization}
\end{align}}\textbf{Reward.} Each agent $i$ is associated with a reward function $r_i(s_i)$ that depends on agent $i$'s state.\footnote{Without too much difficulty, our work can be generalized to the case $r_i$ can depend on $s_i(t),a_i(t), s_{i}(t+1)$. } We will also interpret $r_i$ as a two-dimensional vector in $\R^2$ and we assume that all rewards $r_i$ are bounded above by $\bar{r}$. The total stage reward $r(s)$ for $s=(s_1,\ldots,s_n)$ is the summation of individual reward functions,
\begin{align}
r(s) = \sum_{i=1}^n r_i(s_i). \label{eq:r_decomposable}
\end{align}

\textbf{Local Policy.} In this paper, we focus on local polices. At node $i$, a local policy $\zeta_i:\{0,1\}\rightarrow\{0,1\}$ is a deterministic decision rule that determines the action $a_i(t)$ based on the current state $s_i(t)$ at node $i$. The set of all possible deterministic local policies at node $i$ is denoted as $\Gamma_i$, and $|\Gamma_i|=4$ as there are $4$ distinct maps from $s_i\in \{0,1\}$ to $a_i\in\{0,1\}$. We collectively write all agents' local policies in the tuple $\zeta = (\zeta_1,\ldots,\zeta_n)\in\Gamma_1\times\cdots\times\Gamma_n\triangleq \Gamma$.

\textbf{Stationary Distribution.} In Proposition~\ref{prop:ergodic}, we show that if the individual transition $P_i$ is ergodic in some sense, then the entire MDP is ergodic for localized policies. Due to space limit, the proof of Proposition~\ref{prop:ergodic} is postponed to our online report \cite[Appendix-B]{online}. Throughout this paper, we assume the conditions in Proposition~\ref{prop:ergodic} hold and the stationary distribution exists and is unique.

\begin{proposition}\label{prop:ergodic}
If for each $i$, $\forall s_{\delta_i^1}\in\{0,1\},\zeta_i\in \Gamma_i$, if the Markov Chain on state $s_i$ determined by the transition matrix
$P_i(s_i(t+1)|s_i(t),s_{\delta_i^1}(t) = s_{\delta_i^1}, a_i(t) = \zeta_i(s_i(t))) $
is ergodic, then for the transition matrix \eqref{eq:P_factorization} and for all local policies $\zeta \in \Gamma$, the induced Markov chain is ergodic and has a unique stationary distribution $\pi$. 
\end{proposition}

\textbf{Average Reward.} Given a local policy $\zeta = (\zeta_1,\ldots,\zeta_n)\in\Gamma$, the stationary distribution of the states is determined by Proposition~\ref{prop:ergodic}, which we denote as $\pi(\zeta)\in\R^{2^n}$, where highlight the dependence of the stationary distribution on the policy $\zeta$. We also define $\pi_i(\zeta)$ to be the marginalized stationary distribution for state $s_i$ under policy $\zeta$. Hence, the average reward $R(\zeta)$ under policy $\zeta$ (for any initial state $s(0)$) is

{\small\begin{align}
R(\zeta)& = \lim_{T\rightarrow\infty} \frac{1}{T} \underset{a(t) = \zeta(s(t))}{\mathbb{E}}  \bigg[ \sum_{t=0}^T r(s(t)) | s(0)\bigg] \nonumber \\
&= \underset{s \sim \pi(\zeta)}{\mathbb{E}} r(s)  = \sum_{i=1}^n \underset{s_i \sim \pi_i(\zeta)}{\mathbb{E}} r_i(s_i) = \sum_{i=1}^n r_i^T \pi_i(\zeta). \label{eq:average_total_reward}
\end{align}}We will also define $$R_i(\zeta) = \underset{s_i \sim \pi_i(\zeta)}{\mathbb{E}} r_i(s_i) =  r_i^T \pi_i(\zeta)$$ to be the average reward at node $i$ under policy $\zeta$. 

\textbf{Reward Maximization.} We consider the problem of finding a local policy $\zeta=(\zeta_1,\ldots,\zeta_n)\in\Gamma$ such that the average reward is maximized. 
\begin{align}
     \max_{\zeta\in \Gamma} R(\zeta) \label{eq:opt}
\end{align}

We define tuple $\mathfrak{M}=(n,\mathcal{G},\{P_i,r_i\}_{i=1}^n)$ as a problem instance, where $n$ is the number of agents, $\mathcal{G}$ is the directed tree graph, $P_i$ is the local transition probability, and  $r_i$ is the local reward. 

Since $\Gamma = \Gamma_1\times\Gamma_2\times\cdots\times\Gamma_n$, we have $|\Gamma| = \prod_{i=1}^n|\Gamma_i| = 4^n$, indicating the search space is exponential in $n$. Despite this, in the next two chapters, we exploit the structures in the problem as exhibited in the factored transition probability \eqref{eq:P_factorization} and the additive reward \eqref{eq:r_decomposable} to develop a framework that finds near optimal solutions to \eqref{eq:opt} in time polynomial in $n$ under some assumptions on the problem instance.


\textbf{Notation Conventions. } 
$s_C$ for a tuple $C\subset \mathcal{N}$ means the state tuple $(s_i)_{i\in C}$. In particular, we will use $s_{i:j}$ to denote the state tuple $(s_i,s_{i+1},\ldots, s_j)$; we use $s_{\Delta_i^k}$ for the state tuple on the path from $i$ to $i$'s $k$-hop parent, $s_{\tilde\Delta_i^k}$ for the state tuple on the path from $i$'s $1$-hop parent to $k$-hop parent. We use the same convention for action $a$, policy $\zeta$ and set of policies $\Gamma$. For stationary distribution, we use $\pi_{1:n}$, or simply $\pi$ to denote the stationary distribution for the tuple of all states $(s_1,\ldots,s_n)$, we use $\pi_{C}$ where $C\subset\mathcal{N}$ to denote the marginalized stationary distribution for $(s_i)_{i\in C}$. In particular, $\pi_i$ is the marginalized stationary distribution for node $i$; $\pi_{i:j}$ is the marginalized stationary distribution for nodes $(i,i+1,\ldots,j)$. 

\section{Locality-based Local Policy Search}\label{sec:llps}

 Given problem instance $\mathfrak{M}=(n,\mathcal{G},\{P_i,r_i\}_{i=1}^n)$, we propose a Locality-based Local Policy Search (LLPS) method to approximately find a local policy $\zeta\in\Gamma$ that solves \eqref{eq:opt} approximately. 

\textbf{Overview.} Recall that the objective function $R$ is a summation of the $R_i$'s (cf. \eqref{eq:average_total_reward}). For each $i$, the average reward $R_i(\zeta_1,\zeta_2,\ldots,\zeta_n)$ is a function of the policies of all the nodes and even optimizing $R_i$ is difficult due to the exponentially large search space. In this section, we will show that despite the dependence on the local policies of all nodes, $R_i(\cdot)$ has special structure that makes optimization tractable. Firstly in Section~\ref{subsec:def_approximated_reward}, we define an approximated reward function $\hat{R}^k$ by replacing $R_i$ by an approximated local reward function $\hat{R}_i^k$ that only depends on the local policy of $i$ and $i$'s upto $k$-hop parents, reducing the search space to $O(4^k)$, much more efficient for small $k\ll n$. Secondly in Section~\ref{subsec:approximation_error}, we bound the gap between the true reward function and the approximated reward function. Lastly in Section~\ref{subsec:optimizing}, we optimize the approximated reward function. 

\subsection{Definition of Approximated Reward} \label{subsec:def_approximated_reward}
For each $i$, we consider a truncated MDP by looking at a MDP consisting of $i$ and $i$'s upto $k$-hop parents, while truncating the influence from nodes beyond the $k$-hop parent. In details, the truncated MDP consists of the path from node $i$ to $i$'s $k$-hop parent. The set of nodes on the path is $\Delta_i^k = (i,\delta_i^1,\ldots,\delta_i^k)$. See Figure \ref{fig:truncation} for an illustration. The local transition probabilities of node $(i, \delta_i^1,\delta_i^2,\ldots,\delta_i^{k-1})$ will be the same as the original model; however for the $k$-hop parent $\delta_i^k$, we will simply draw $s_{\delta_i^k}(t)$ uniformly from $\{0,1\}$ for all $t$, independent from everything else. More formally, the transition probability for state tuple $s_{\Delta_i^k}$ under action $a_{\Delta_i^k}$ is given by
\begin{align*}
&P(s_{\Delta_i^k}(t+1)|s_{\Delta_i^k}(t),a_{\Delta_i^k}(t)) \\
&= P_i(s_i(t+1)|s_i(t),s_{\delta_i^1}(t),a_i(t)) \\
&\quad \times P_{\delta_i^1}(s_{\delta_i^1}(t+1)| s_{\delta_i^1}(t), s_{\delta_i^2}(t), a_{\delta_i^1}(t) ) \\
&\quad \times \cdots\times  P_{\delta_i^{k-1}}(s_{\delta_i^{k-1}}(t+1)| s_{\delta_i^{k-1}}(t), s_{\delta_i^{k}}(t), a_{\delta_i^{k-1}}(t) ) \\
&\quad \times Uniform(s_{\delta_i^{k}}(t+1))
\end{align*}
which factorizes in the same way as the original model, except that the last factor $Uniform(s_{\delta_i^{k}}(t+1))$ is the uniform probability mass function on $\{0,1\}$, showing $s_{\delta_i^{k}}(t+1)$ is drawn from the uniform distribution independent of everything else. 
By forcing $s_{\delta_i^k}(t)$ to be drawn independently from a uniform distribution, we are effectively truncating the influence from all upper-stream nodes, and focus on a MDP of $k+1$ nodes. We will refer to the truncated model as $\hat{\mathfrak{M}}_i^k$.


For the truncated MDP $\hat{\mathfrak{M}}_i^k$, given local policy $(\zeta_i,\zeta_{\delta^1_i},\ldots,\zeta_{\delta^k_i}) = \zeta_{\Delta_i^{k}}$, the stationary distribution for state tuple $s_{\Delta_i^k}$ is determined, and we define the marginalized stationary distribution at node $i$ as $\hat{\pi}_i^k(\zeta_{\Delta^{k}_i}) $, which is a function of the policies of nodes in $\Delta_i^k$. Further, given $\hat{\pi}_i^k(\zeta_{\Delta^{k}_i}) $, we define the the average reward at $i$ in the truncated model as the following, $$\hat{R}_i^k(\zeta_{\Delta^{k}_i}) = \underset{s_i\sim \hat{\pi}_i^k( \zeta_{\Delta^{k}_i})}{\mathbb{E}} r_i(s_i)=r_i^T \hat{\pi}_i^k(\zeta_{\Delta^{k}_i}) $$

The \emph{approximated reward function} $\hat{R}^k(\zeta_1,\ldots,\zeta_n)$ is defined to be the summation of the $\hat{R}_i^k(\zeta_{\Delta^{k}_i})$. 
\begin{align}
    \hat{R}^k(\zeta_1,\ldots,\zeta_n) = \sum_{i=1}^n \hat{R}_i^k(\zeta_{\Delta^{k}_i})   \label{eq:approx_reward}
\end{align}
\begin{figure}
    \centering
    \includegraphics[width=0.8\columnwidth]{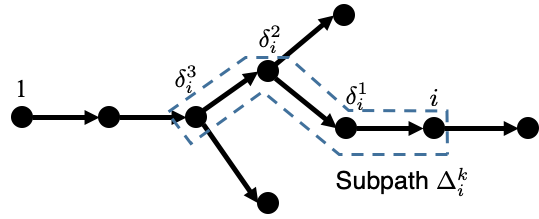}
    \caption{Illustration of truncated model at node $i$ for $k=3$. }
    \label{fig:truncation}
\end{figure}

\subsection{Approximation Error } \label{subsec:approximation_error}
The idea behind the above truncated model and approximated reward is that, intuitively, the policy at a node far away from $i$ should not have that much effect on $i$. Therefore, by looking at the truncated model instead of the full model, we should get a good approximation of $R_i$. In the following, the above intuition will be formally stated as ``fast decaying property'' in Definition~\ref{def:fast_decaying}, and we show in Lemma~\ref{lem:bounded_approx_error} that under the fast decaying property, the approximation error will decay exponentially in $k$. In this section, we will assume the fast decaying property holds. We will study under what conditions the fast decaying property holds in Section~\ref{sec:fast_decaying}. 

\begin{definition}\label{def:fast_decaying}
A Multi-Agent Markov Model $\mathfrak{M}$ has $(C,\rho)$-fast decaying property if for any $i$ and any local policy $\zeta=(\zeta_1,\ldots,\zeta_n)\in\Gamma$, 
$$ \Vert \pi_i(\zeta) - \hat\pi_i^k(\zeta_{\Delta^{k}_i})\Vert_1 \leq C\rho^k$$
\end{definition}
The fast decaying property directly leads to the following bound on the approximation error.
\begin{lemma}\label{lem:bounded_approx_error}
If the Multi-Agent Markov Model $\mathfrak{M}$ has $(C,\rho)$-fast decaying property, then we have for all $\zeta=(\zeta_1,\ldots,\zeta_n)$, 
$$\frac{1}{n}|R(\zeta) - \hat{R}^k(\zeta)|\leq  C\bar{r} \rho^k $$
\end{lemma}
\begin{proof} Fixing $\zeta=(\zeta_1,\ldots,\zeta_n)$, then by the definition of $\hat{R}_i^k(\zeta_{\Delta^k_i})$, we have
\begin{align*}
    |\hat{R}_i^k(\zeta_{\Delta_i^k} ) - R_i(\zeta)| &= |r_i^T (\pi_i(\zeta) - \hat{\pi}_i^k(\zeta_{\Delta_i^k} ))|\\
    &\leq \Vert r_i\Vert_\infty \Vert \pi_i(\zeta) - \hat{\pi}_i^k(\zeta_{\Delta_i^k})\Vert_1 \leq \bar{r} C\rho^k
\end{align*}
Therefore, 
\begin{align*}
    \frac{1}{n}|R(\zeta) - \hat{R}^k(\zeta)| 
    &\leq \frac{1}{n} \sum_{i=1}^n | \hat{R}_i^k(\zeta_{\Delta_i^k}) - R_i(\zeta) | \leq C\bar{r} \rho^k
\end{align*}
\end{proof}

\subsection{Optimizing Approximated Reward Function}\label{subsec:optimizing}
The approximated reward function \eqref{eq:approx_reward} is a summation of individual $\hat{R}_i^k(\cdot)$, each of which only depend on policies at $k+1$ nodes. Maximization of such functions is equivalent to the MAP (Maximum a Posteriori) problem in graphical models with tree structures \cite{wainwright2008graphical}, and can be efficiently solved by many (mutually-equivalent) methods, like belief propagation, max-product, elimination method, or dynamic programming, cf. \cite[Sec. 2.5]{wainwright2008graphical} for a review. Here we provide a dynamic programming based algorithm for maximizing the approximated reward function.

The algorithm LLPS is provided in Algorithm~\ref{algorithm:optimize}. In the algorithm, BFS means Breadth-First Search (BFS) Ordering \cite{bfs} of the nodes, in which the root node comes first and parents must precede children, and leaf nodes come at last; reverse BFS means the reverse of a BFS ordering, where children always come before parents. 


\begin{algorithm}\caption{LLPS: Locality-based Local Policy Search}\label{algorithm:optimize}
	\begin{algorithmic}[1]
	    \For{$i$ traverses the tree in reverse BFS ordering}
	        \State Collect $V_j(\zeta_{\tilde\Delta_j^k})$ for all children $j\in c_i$.
	        \For{$\zeta_{\tilde\Delta^k_i} \in \Gamma_{\tilde\Delta^k_i}$}
	             \State $V_i(\tilde\Delta^k_i) \gets $ \par \hskip\algorithmicindent  $\underset{\zeta_i\in \Gamma_i }{\max} \hat{R}_i^k(\zeta_i,\zeta_{\tilde\Delta^k_i}) + \underset{j\in c_i}{\sum} V_j(\zeta_i,\zeta_{\tilde\Delta^{k-1}_i}) $
	        \EndFor
	    \EndFor

	 \State $\zeta_1^* \gets \underset{\zeta_1\in\Gamma_1}{\arg\max} \hat{R}_1^k(\zeta_1) + \underset{j\in c_1}{\sum} V_j(\zeta_1 )  $
	 \For{$i$ traverses the tree except the root in BFS ordering }
	    \State $\zeta_i^*\gets \underset{\zeta_i\in \Gamma_i }{\arg\max} \hat{R}_i^k(\zeta_i,\zeta^*_{\tilde\Delta^k_i}) + \underset{j\in c_i}{\sum} V_j(\zeta_i,\zeta^*_{\tilde\Delta^{k-1}_i})  $
	 \EndFor
	 \State \textbf{Output }$\zeta^*=(\zeta_1^*,\ldots,\zeta_n^*)$.
	 \end{algorithmic}
\end{algorithm}
\begin{figure}
    \centering
    \includegraphics[width=0.8\columnwidth]{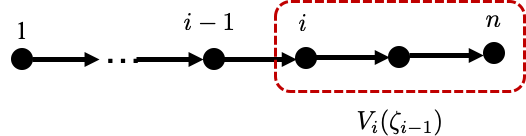}
    \caption{The case when the graph is a line.}
    \label{fig:linecase}
\end{figure}
\textit{Illustration on a line.} For ease of presentation and explanation, we provide the special case of the algorithm for a line graph with $k=1$, where the nodes are labelled in the natural ordering with the root node being $1$, and the parent of $i$ is $i-1$, as illustrated in Figure~\ref{fig:linecase}. In this case, the approximated reward function can be written as,
\begin{align}
\hat{R}^1(\zeta) = \hat{R}_1^1(\zeta_1) + \hat{R}_2^1(\zeta_1,\zeta_2) + \ldots \hat{R}_n^1(\zeta_{n-1},\zeta_n) \label{eq:approx_reward_line}
\end{align}
and we provide the algorithm below, in two steps.

\noindent\textit{Step 1.} For $i = n$ to $2$, construct the following table $V_i(\cdot)$. For every $\zeta_{i-1} \in \Gamma_{i-1}$, set $V_i(\zeta_{i-1}) = \underset{\zeta_i\in \Gamma_i }{\max} \hat{R}_i^1(\zeta_{i-1},\zeta_i) + V_{i+1}(\zeta_i) $, where $V_{i+1}$ is the table constructed for node $i+1$ and $V_{n+1}(\cdot)$ is treated as $0$.

\noindent\textit{Step 2.} Compute $\zeta_1^*=   \underset{\zeta_i\in \Gamma_i }{\arg\max} \hat{R}_1^1(\zeta_1) + V_{2}(\zeta_1)$. Then, for $i=2$ to $n$, compute $\zeta_i^* = \underset{\zeta_i\in \Gamma_i }{\arg\max} \hat{R}_i^k(\zeta_{i-1}^*,\zeta_i) +  V_{i+1}(\zeta_i)$. 
	   

The algorithm is similar to the backtracking method in dynamic programming \cite{bertsekas2007dpbook}, except here we do induction on a spatial quantity $i$ instead of time. The key to understanding the algorithm is $V_i(\zeta_{i-1})$. In fact, $V_i(\zeta_{i-1})$ is the optimal reward value of the path from $i$ to the leaf node $n$, if the policy at $i-1$ is $\zeta_{i-1}$. More formally, $V_i$ satisfies the following Lemma.
\begin{lemma}\label{lem:line_backward}
In LLPS for $k=1$ on a line graph, for $i\geq 2$,
\begin{align*}
V_i(\zeta_{i-1}) &= \max_{\zeta_{i},\ldots,\zeta_n} \hat{R}_i^1(\zeta_{i-1},\zeta_i) + \ldots+ \hat{R}_n^1(\zeta_{n-1},\zeta_n)  
\end{align*}
\end{lemma}
\begin{proof} 
The lemma is true for $i=n$ by the definition of $V_{n}(\cdot)$, and it is true for $i<n$ by a simple induction argument. 
\end{proof}
With this lemma, we can easily show that
\begin{align*}
    \zeta_1^* &= \underset{\zeta_1}{\arg\max} \underset{\zeta_2,\ldots,\zeta_n}{\max}  \hat{R}_1^1(\zeta_1) + \hat{R}_2^1(\zeta_1,\zeta_2) + \ldots \hat{R}_n^1(\zeta_{n-1},\zeta_n) \\
    &= \underset{\zeta_1}{\arg\max} \underset{\zeta_2,\ldots,\zeta_n}{\max} \hat R^1(\zeta_1,\ldots,\zeta_n)
\end{align*}
and 
\begin{align*}
   (\zeta_1^*,\ldots, \zeta_i^*) = \underset{(\zeta_1,\ldots,\zeta_i)}{\arg\max} \underset{\zeta_{i+1},\ldots,\zeta_n}{\max} \hat{R}^1(\zeta_1,\ldots,\zeta_i,\zeta_{i+1},\ldots,\zeta_n)
\end{align*}
showing $(\zeta_1^*,\ldots, \zeta_n^*)$ is indeed a maximizer of $\hat{R}^1$. 

The above arguments also extends to the general case in Algorithm~\ref{algorithm:optimize}, except that the traversing order is (reverse) Breadth First Search, the children of node $i$ may not be $i+1$, and each $V_i$ may depend on the states on $i$'s $1$-hop to $k$-hop parents $\tilde\Delta_i^k$, not just $1$-hop parent. As a result of these reasoning, the algorithm LLPS finds a maximizer of $\hat{R}^k$ for the general tree-graph case. Now we discuss the computational complexity of implementing LLPS.

\vspace{4pt}
 \noindent\textbf{Complexity of evaluating $\hat{R}_i^k$.} By definition, $\hat{R}_i^k(\zeta_{\Delta^{k}_i})=r_i^T \hat{\pi}_i^k(\zeta_{\Delta^{k}_i}) $. Evaluation of $\hat{R}_i^k(\zeta_{\Delta^{k}_i})$ requires computing the stationary distribution of the truncated model, whose state space size is $2^{k+1}$, and then marginalize it to $\hat{\pi}_i^k(\zeta_{\Delta^{k}_i})$, and finally do a simple inner product. The bulk of the computation is calculating the stationary distribution, which takes at most cubic time in state-space size. Therefore, the evaluation of $\hat{R}_i^k$ takes time $O(2^{3k})$.

\vspace{4pt}
\noindent\textbf{Time complexity of LLPS.} For each $\zeta_{\tilde\Delta_i^k}$, the calculation of $V_i(\zeta_{\tilde\Delta_i^k}) $ takes time $O(2^{3k}+d)$, where $d$ is the maximum degree of the graph. As a result, the implementation of the algorithm takes time $O(n  4^k (2^{3k}+d))$, which is linear in $n$ for fixed $k$ and $d$.  

\vspace{4pt}
\noindent\textbf{Space complexity of LLPS.} For each $i$, a table $V_i(\cdot)$ of size $O(4^k)$ needs to be maintained.  Therefore the space complexity is $O(n 4^k)$.

Wrapping up the above results, as well as the approximation gap in Lemma~\ref{lem:bounded_approx_error}, we have the following straightforward proposition.
\begin{proposition}\label{prop:main}
LLPS finds a maximizer $\zeta^*=(\zeta_1^*,\ldots,\zeta_n^*)$ of $\hat{R}^k(\zeta_1,\ldots,\zeta_n)$ within time $O(n  4^k (2^{3k}+d))$. If the problem instance has $(C,\rho)$-fast decaying property, then $\zeta^*$ is a $C\rho^k\bar{r}$ approximate optimal solution in the sense that,
$$\frac{1}{n}\big[\max_{\zeta\in\Gamma} R(\zeta) - R(\zeta^*) \big]\leq C\bar{r}\rho^k.$$
\end{proposition}

\noindent\textbf{Tradeoff between computation and accuracy.} It can be seen in Proposition~\ref{prop:main} that computational complexity increases exponentially in $k$, whereas the optimality gap of the solution decays exponentially $k$. This suggests a tradeoff between computational complexity and desired accuracy level of the solution. In Section~\ref{sec:case_study}, we will discuss the tradeoff in more detail through numerical analysis.

\section{Fast Decaying Property}\label{sec:fast_decaying}
In this section, we provide conditions under which the $(C,\rho)$-fast decaying holds. To state our theoretic results, we parameterize the transition probabilities as follows, where the ordering of the rows and columns is such that the first row (column) corresponds to state $0$, and the second row (column) corresponds to state $1$.

{\small\begin{align*}
P_i(s_i(t+1)  | s_i(t) , s_{\delta_i^1}(t) =0,a_i(t)=0) &= \left[\begin{array}{ll}
e_i & 1-e_i\\
f_i & 1 - f_i
\end{array}\right]\\
P_i(s_i(t+1)  | s_i(t) , s_{\delta_i^1}(t) =1,a_i(t)=0) &= \left[\begin{array}{ll}
g_i & 1-g_i\\
h_i & 1 - h_i
\end{array}\right] \\
P_i(s_i(t+1)  | s_i(t) , s_{\delta_i^1}(t) =0,a_i(t)=1) &= \left[\begin{array}{ll}
e_i' & 1-e_i'\\
f_i' & 1 - f_i'
\end{array}\right]\\
P_i(s_i(t+1)  | s_i(t) , s_{\delta_i^1}(t) =1,a_i(t)=1) &= \left[\begin{array}{ll}
g_i' & 1-g_i'\\
h_i' & 1 - h_i'
\end{array}\right] 
\end{align*}}For the root node $i=1$ who has no parent, we will simply set $e_1=g_1$, $f_1 = h_1$, $e'_1 = g'_1$ and $f_1' = h_1'$. Our results can be stated as follows,
\begin{lemma} \label{lem:fast_decaying}
Given problem instance $\mathfrak{M}$. If for all $i$, the parameters satisfy,
{\small\begin{align*}
    e_i - f_i &=g_i - h_i \in(-1,1),\quad  e_i - f_i'&=g_i - h_i' \in(-1,1)\\
    e_i' - f_i &=g_i' - h_i \in(-1,1),\quad  e_i' - f_i' &=g_i' - h_i' \in(-1,1)
\end{align*}}and further if 
{\small\begin{align*}
    \max\Bigg\{&|\frac{e_i - g_i}{1- (e_i-f_i)}|, |\frac{e_i - g_i}{1-(e_i-f_i')}|, \\
&|\frac{e_i' - g_i'}{1-(e_i'-f_i)}|,|\frac{e_i' - g_i'}{1-(e_i'-f_i')}| \Bigg\}\leq \rho\in(0,1)
\end{align*}}Then, the problem instance is $(2,\rho)$-fast decaying. 
\end{lemma}

The reason we impose the assumption on the transition probability parameters is that it greatly simplifies the proof, which we will provide in Section~\ref{subsec:proof_fast_decaying}. However, we conjecture that the fast decaying property holds for much general range of parameters. We will conduct numerical simulations in Section~\ref{subsec:numerical_fast_decaying} that certifies fast decaying property for randomly generated transition probability parameters.

\subsection{Proof of Lemma~\ref{lem:fast_decaying}}\label{subsec:proof_fast_decaying}

We notice that under the tree dependence structure, fixing any local policy $\zeta=(\zeta_1,\ldots,\zeta_n)$, the states along the path from the root to node $i$ form a self-complete Markov chain, and the marginalized stationary distribution $\pi_i(\zeta)$ at node $i$ only depends on the Markov Chain on the path. Further, the truncated model at node $i$ is a subpath of the path from root to $i$, where the nodes beyond $i$'th $k$-hop parent are ``truncated''. Therefore, to study the fast decaying property it suffices to study the case where the dependence graph is a line rooted at $1$ and with leaf node $i$. We label the nodes with the nature ordering on the line.

In the line structure, the parent of node $j$ is $j-1$ (for $j\geq 2$). 
We fix a local policy $\zeta=(\zeta_1,\ldots,\zeta_i)$ and from now on, drop the dependence on policy. Let the local transition probabilities induced by the policy be, 
{\small
\begin{subeqnarray}\label{eq:p_j_parameterization}
P_j(s_j(t+1)  | s_j(t) , s_{j-1}(t) =0) &=& \left[\begin{array}{ll}
\alpha_j & 1-\alpha_j\\
\beta_j & 1 - \beta_j
\end{array}\right]\\
P_j(s_j(t+1)  | s_j(t) , s_{j-1}(t) =1) &=& \left[\begin{array}{ll}
\gamma_j & 1-\gamma_j\\
\omega_j & 1 - \omega_j
\end{array}\right] \\
P_1(s_1(t+1)  | s_1(t)  ) &= &\left[\begin{array}{ll}
\alpha_1 & 1-\alpha_1\\
\beta_1 & 1 - \beta_1
\end{array}\right] 
\end{subeqnarray}}where we also introduce $\gamma_1 = \alpha_1$ and $\omega_1 = \beta_1$ for notational consistency. The value of $(\alpha_j,\beta_j,\gamma_j,\omega_j)$ depends on the local policy $\zeta_j$ and the local transition parameters $(e_j,f_j,g_j,h_j,e_j',f_j',g_j',h_j')$. For example, if $\zeta_j$ is such that action $a_j=0$ regardless of $s_j$, then $(\alpha_j,\beta_j,\gamma_j,\omega_j) = (e_j,f_j,g_j,h_j)$; if the $\zeta_j $ is such that $a_j=0$ when $s_j=0$, $a_j=1$ when $s_j=1$, then $(\alpha_j,\beta_j,\gamma_j,\omega_j) = (e_j,f_j',g_j,h_j')$. Under the assumptions in Lemma~\ref{lem:fast_decaying}, it is easy to check that regardless of the policy $\gamma_j$, there exists some $\mu_j\in(-1,1)$ s.t. $(\alpha_j,\beta_j,\gamma_j,\omega_j)$ satisfies,
\begin{align}
    \alpha_j - \beta_j = \gamma_j - \omega_j &=\mu_j  \\
    \Big|\frac{\alpha_j - \gamma_j}{1-\mu_j} \Big|&\leq \rho \label{eq:proof_fast_decaying:rho}
\end{align}We denote the stationary distribution of the model as $\pi_{1:i}$, and the marginalized distribution at node $j$ as $\pi_j$ for $1\leq j\leq i$. 

The statement of Lemma~\ref{lem:fast_decaying} bounds marginalized stationary distribution $\pi_i$ and that of the truncated model. Recall that, the truncated model is on state tuple $s_{i-k:i}$, i.e. on the path from $i$ to $i$'s $k$-hop parent, node $i-k$. Let the local transition probabilities be $\hat{P}_j(\cdot)$ for $i-k\leq j\leq i$, and let them be parameterized by $(\hat\alpha_j,\hat\beta_j,\hat\gamma_j,\hat\omega_j)_{i-k\leq j\leq i}$ in the same fashion as \eqref{eq:p_j_parameterization}. By the way the truncated model is constructed, we have $\hat{\alpha}_j = \alpha_j, \hat{\beta}_j = \beta_j, \hat\gamma_j = \gamma_j, \hat\omega_j=\omega_j$ for $i-k+1\leq j\leq i$; and for $j=i-k$,
$$\hat{\alpha}_{i-k} = \hat{\beta}_{i-k} = \hat{\gamma}_{i-k}= \hat{\omega}_{i-k} = \frac{1}{2}.$$
Setting $\hat{\mu}_j = \mu_j$ for $i-k< j \leq i$ and $\hat{\mu}_{i-k} = 0$, we have for $i-k\leq j\leq i$,
\begin{align}
        \hat\alpha_j - \hat\beta_j = \hat\gamma_j - \hat\omega_j &=\hat\mu_j\in(-1,1)  
\end{align}
Let the stationary distribution of truncated model be $\hat{\pi}_{i-k:i}$, and let the marginalized stationary distribution at node $j$ be $\hat{\pi}_j$ for $i-k\leq j\leq i$. 

We now provide a recursive formula on the marginalized stationary distribution that works for both the original model and the truncated model.


\begin{lemma}\label{lem:a_i_sensitivity}
For any Markov chain on state tuple $s_{1:i}=(s_1,\ldots,s_i)$ that factorizes in in the following way,
{\small\begin{align*}
    &P(s_{1:i}(t+1)|s_{1:i}(t)) \\
    &= P_1(s_1(t+1)|s_1(t)) \prod_{j=2}^i P_j(s_j(t+1)|s_j(t),s_{j-1}(t))
\end{align*}}where $P_j$ is parameterized by,
{\small
$$P_j(s_j(t+1) = \cdot | s_j(t) = \cdot, s_{j-1}(t) =0) = \left[\begin{array}{ll}
\alpha_j & 1-\alpha_j\\
\beta_j & 1 - \beta_j
\end{array}\right]$$
$$P_j(s_j(t+1) = \cdot | s_j(t) = \cdot, s_{j-1}(t) =1) = \left[\begin{array}{ll}
\gamma_j & 1-\gamma_j\\
\omega_j & 1 - \omega_j
\end{array}\right]$$
$$P_1(s_1(t+1) = \cdot | s_1(t) = \cdot ) = \left[\begin{array}{ll}
\alpha_1 & 1-\alpha_1\\
\beta_1 & 1 - \beta_1
\end{array}\right]$$}Assume $\alpha_ j- \beta_j = \gamma_j - \omega_j = \mu_j\in(-1,1)$. For $1\leq j\leq i$, let $b_j$ be the probability of $s_j=1$ under the stationary distribution. Then, for any $1\leq k\leq i-1$, 
\begin{align*}
 b_i = b_{i-k}\prod_{j=i-k+1}^i \frac{\alpha_j - \gamma_j}{1 - \mu_{j}} + \sum_{j=i-k+1}^i \frac{1-\alpha_j}{1-\mu_j} \prod_{\ell=j+1}^i \frac{\alpha_\ell - \gamma_\ell}{1-\mu_\ell} 
 \end{align*}
\end{lemma}
Recall that $\pi_j$ is the stationary distribution of $s_j$ in the original model. We use the notation in Lemma~\ref{lem:a_i_sensitivity} to write it as $\pi_j = [1-b_j,b_j]$. Similarly, we write $\hat{\pi}_j = [1 - \hat{b}_j,\hat{b}_j]$ for the truncated model. Then Lemma~\ref{lem:a_i_sensitivity} shows that for the original model,
\begin{align*}
 b_i = b_{i-k}\prod_{j=i-k+1}^i \frac{\alpha_j - \gamma_j}{1 - \mu_{j}} + \sum_{j=i-k+1}^i \frac{1-\alpha_j}{1-\mu_j} \prod_{\ell=j+1}^i \frac{\alpha_\ell - \gamma_\ell}{1-\mu_\ell} 
 \end{align*}
and for the truncated model,
\begin{align*}
 \hat{b}_i = \hat{b}_{i-k}\prod_{j=i-k+1}^i \frac{\hat\alpha_j - \hat\gamma_j}{1 - \hat\mu_{j}} + \sum_{j=i-k+1}^i \frac{1-\hat\alpha_j}{1-\hat\mu_j} \prod_{\ell=j+1}^i \frac{ \hat\alpha_\ell-\hat\gamma_\ell}{1-\hat\mu_\ell} 
 \end{align*}
Since for $j\geq i-k+1$, $(\alpha_j,\beta_j,\gamma_j,\omega_j,\mu_j) = (\hat\alpha_j,\hat\beta_j,\hat\gamma_j,\hat\omega_j,\hat\mu_j)$, we have
\begin{align*}
    |b_i - \hat b_i| = |b_{i-k} - \hat b_{i-k}| \prod_{j=i-k+1}^i \Big|\frac{\alpha_j - \gamma_j}{1 - \mu_{j}}\Big|\leq \rho^k
\end{align*}
where we have used \eqref{eq:proof_fast_decaying:rho}. As a result,
$\Vert\pi_i - \hat\pi_i\Vert_1\leq 2\rho^k$, which concluces the proof. 

\subsection{An Empirical Study of Fast Decaying Property}\label{subsec:numerical_fast_decaying}
We note that Lemma~\ref{lem:fast_decaying} requires the parameters $(e_i,f_i,g_i,h_i,e_i',f_i',g_i',h_i')$ to satisfy certain equality constraint. However, we conjecture that the fast decaying property holds for much more general parameters. To this end, we do Monte-Carlo simulations to test our conjecture. We fix the graph to be a line consisting of $10$ nodes. For each run, we generate all parameters uniformly from $[0,1]$, and also select $\zeta_i$ uniformly random from one of the $4$ possible maps from $s_i\in\{0,1\}$ to $a_i\in\{0,1\}$. We then compute the marginalized stationary distribution at the leaf node $\pi_{10}$, as well the the marginal distribution of the truncated model $\hat\pi_{10}^k$, for $k=1,\ldots,8$, and plot the gap $\Vert \pi_{10} - \hat\pi_{10}^k\Vert_1$ as a function of $k$. We do $50$ runs with random parameters and policies. The results are given in Figure~\ref{fig:fast_decaying}.  The figure confirms that the error $\Vert \pi_{10} - \hat\pi_{10}^k\Vert_1$ decays exponentially fast in $k$. 

\begin{figure}[t]
		\centering
		\includegraphics[width=0.9\columnwidth]{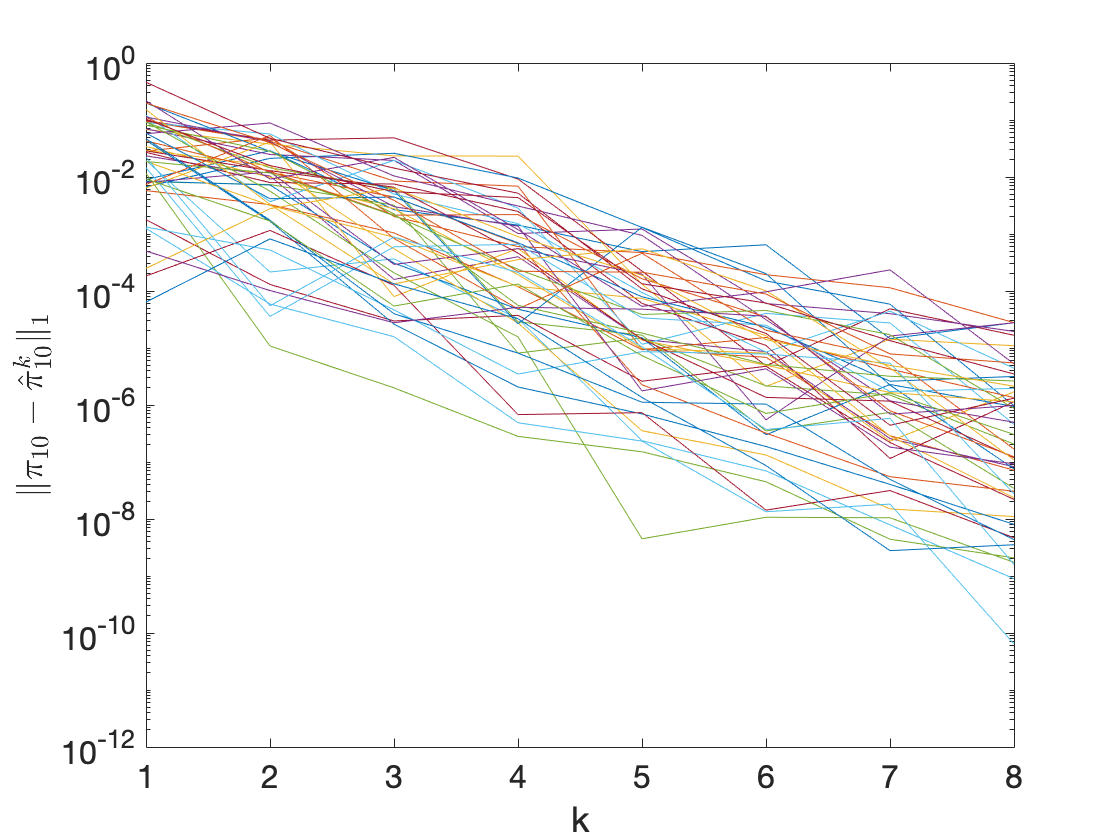} 

	\caption{Monte Carlo simulations showing $\Vert \pi_{10} - \hat\pi_{10}^k\Vert_1$ as a function of $k$. }\label{fig:fast_decaying}
\end{figure}
\section{Case Studies}\label{sec:case_study}
In this section, we conduct simulations on our algorithm. We let the dependence graph be the $9$-node graph in Fig.~\ref{fig:truncation}. We draw the local transition probability parameters $(e_i,f_i,g_i,h_i,e_i',f_i',g_i',h_i')$ independently from the uniform distribution on $[0,1]$, and we draw each component of the reward vector $r_i$ from the uniform distribution on $[0,1]$. For a model of this size, we are able to find the exact optimal solution through exhaustive search. We then run our algorithm LLPS for $k=1,2,\ldots,7$. In Table~\ref{table:results}, we provide the reward achieved by the optimal local policy (obtained through exhaustive search) and the reward achieved by the local policy output by LLPS for different $k$, as well as the optimality gap and the running time\footnote{All numerical simulations are done on MATLAB2018b using a MacBook Pro Early 2015 Model with 2.7GHz Intel Core i5.} of the respective algorithm.  As shown in Table~\ref{table:results}, our algorithm is several magnitudes faster than the exhaustive search, and increasing $k$ also causes our algorithm to be slower. Further, the reward achieves by our algorithm is very close to the optimal reward for $k=1,2$, and our algorithm achieves the exact optimal reward for $k\geq 3$. One possible explanation is that for large $k$, the gap between the approximate reward and the true reward, becomes smaller than the gap of the optimal reward and the reward of the second best policy. As a result, the optimizer of the approximate reward and the true reward function coincide.

\begin{table}

	\begin{center}
	\caption{Reward and running time results for exhaustive search and LLPS.}\label{table:results}

	\begin{tabular}{c|ccc}
\textbf{Algorithm}  & \textbf{Reward}  & \textbf{Optimality Gap} & \textbf{Running Time} \\
		\hline
		Exhaustive Search& 4.2578&0 &1804.9s\\
		\hline
		  LLPS $k=1$ & 4.2123 & 0.0456&0.6049s\\
		  LLPS $k=2$ & 4.2563 &0.0016 &0.2558s\\ 
		  LLPS $k=3$ & 4.2578 &0 & 1.6037s\\
		LLPS $k=4$ & 4.2578 &0 & 5.9123s\\
		  LLPS $k=5$ & 4.2578 & 0& 17.9991s\\
		  LLPS $k=6$ & 4.2578 &0 & 46.1841s\\
		  LLPS $k=7$ & 4.2578 &0 & 47.3225s
	\end{tabular}
\end{center}
\end{table}

\section{Conclusion}
We study a multi-agent MDP problem with an exponentially large state and action space, and we aim to find the optimal local policy when the MDP has local dependence structure. We propose a Locality-based Local Policy Search method and show it can efficiently find a near-optimal local policy when the MDP has a ``fast-decaying property''. We show the fast decaying property holds when the dependence structure is a one-directional tree and the transition probabilities satisfy certain assumptions. However, our Monte-Carlo simulations in~\ref{subsec:numerical_fast_decaying} show that the fast decaying property appears to hold under very general conditions. Future work includes gaining deeper understanding on the fast decaying property and show the fast decaying property holds under more general conditions on the parameters and for more general dependence graphs. Further, we are interested in utilizing the fast decaying property to design model-free learning algorithms with provable near-optimal guarantees.
\bibliographystyle{IEEEtran}
\bibliography{refs}

\appendix

\subsection{Proof of Lemma~\ref{lem:a_i_sensitivity}}\label{subsec:appendix_proof_sensitivity}
Firstly, we note that for any $1\leq\ell\leq i$, the states $(s_1,\ldots,s_\ell)$ form a self-complete Markov chain, with transition matrix $P_{1:\ell}$ given by $P_{1:\ell}(s_{1:\ell}(t+1)|s_{1:\ell}(t)) 
    =P_1(s_1(t+1)|s_1(t)) \prod_{j=2}^\ell P_j(s_j(t+1)|s_j(t),s_{j-1}(t))$
and we denote the stationary distribution as $\pi_{1:\ell}\in\R^{2^\ell}$.

For any $1\leq\ell\leq i$, the rows and columns of matrix $P_{1:\ell}\in\R^{2^\ell\times 2^\ell}$, as well as the entries of the stationary distribution $\pi_{1:\ell}$, are indexed by state tuple $(s_1,s_2,\ldots,s_\ell)\in\{0,1\}^\ell$. We use the following ordering that maps each state tuple $(s_1,s_2,\ldots,s_\ell)$ to a row (column) number in $\{1,\ldots,2^\ell\}$.
$$\{0,1\}^\ell\ni  (s_1,s_2,\ldots,s_\ell) \mapsto 1+\sum_{j=1}^\ell 2^{j-1} s_j \in \{1,\ldots,2^\ell\}$$
The above ordering means that, for the $2^\ell$ rows (columns), designate the upper (left) half to $s_\ell=0$ and the lower (right) half to $s_\ell=1$; within each of the halves, designate the upper (left) quarter to $s_{\ell-1}=0$ and the lower (right) quarter to $s_{\ell-1}=0$. Continue this procedure we can uniquely identify each row (column) with a state tuple $(s_1,\ldots,s_\ell)$.

We divide $P_{1:i-1}$ into the upper half and lower half as in the following.
$$P_{1:i-1} = \left[\begin{array}{ll}
P_{1:i-1}^-\\
P_{1:i-1}^+
\end{array} \right] $$ 
It is easy to check,
{\small$$P_{1:i} = \left[\begin{array}{ll}
\alpha_{i} P_{1:i-1}^- & (1-\alpha_{i})  P_{1:i-1}^- \\
\gamma_{i}  P_{1:i-1}^+ & (1-\gamma_{i})  P_{1:i-1}^+\\
\beta_{i}  P_{1:i-1}^- & (1-\beta_{i})  P_{1:i-1}^-\\
\omega_{i}  P_{1:i-1}^+ & (1-\omega_{i})  P_{1:i-1}^+
\end{array} \right].$$}We do the following similarity transform.
{\small
\begin{align*}
P_{1:i} &=	\overbrace{ \left[\begin{array}{cc}
I & 0\\
I & I
\end{array}\right] }^{:=T} {\left[ \begin{array}{cc}
		    \left[\begin{array}{ll}
			P_{1:i-1}^-\\
			P_{1:i-1}^+
		\end{array} \right] &
		 \left[\begin{array}{ll}
			(1-\alpha_{i}) P_{1:i-1}^-\\
			(1-\gamma_{i} ) P_{1:i-1}^+
		\end{array} \right] \\
		0  & 		  \left[\begin{array}{ll}
			(\alpha_{i} - \beta_{i}) P_{1:i-1}^-\\
			(\gamma_{i} -\omega_{i}) P_{1:i-1}^+
		\end{array} \right] 
	\end{array} \right] }\\
	&\quad \times  \left[\begin{array}{cc}
I & 0\\
-I & I
	\end{array}\right]\\
	&=	T \overbrace{\left[ \begin{array}{cc}
	P_{1:i-1} &
		\left[\begin{array}{ll}
		(1-\alpha_{i}) P_{1:i-1}^-\\
		(1-\gamma_{i} ) P_{1:i-1}^+
		\end{array} \right] \\
		0  & 		  \mu_{i} P_{1:i-1}
		\end{array} \right] }^{:=\tilde{P}_{1:i}}  T^{-1} = T \tilde{P}_{1:i} T^{-1}
	\end{align*}}As such, $P_{1:i}$ is similar to $\tilde{P}_{1:i}$. If row vector $\pi$ is an left eigenvector of $\tilde{P}_{1:i}$ with eigenvalue $\lambda$, then $\pi T^{-1} P_{1:i} = \pi T^{-1} T \tilde{P}_{1:i} T^{-1} = \lambda \pi T^{-1} $, i.e. $ \pi T^{-1} $ is an left eigenvector of $P_{1:i}$ with eigenvalue $\lambda$. With this relation, we can study the stationary distribution of $P_{1:i}$ through the eigenvector of $\tilde{P}_{1:i}$. Let $\pi_{1:i-1}$ be the stationary distribution of $P_{1:i-1}$. Let $$\pi_{1:i}^+ = \pi_{1:i-1}  \left[\begin{array}{ll}
	(1-\alpha_{i}) P_{1:i-1}^-\\
	(1-\gamma_{i} ) P_{1:i-1}^+
	\end{array} \right]  (I - \mu_{i}P_{1:i-1})^{-1}  \in \R^{2^{i-1}}$$
	Then, let $\tilde{\pi}_{1:i} = [{\pi}_{1:i-1}, \pi_{1:i}^+]\in\R^{2^i}$. We can verify
	{\small
	\begin{align*}
	\tilde{\pi}_{1:i} \tilde{P}_{1:i} &=  [{\pi}_{1:i-1},\pi_{1:i}^+] \left[ \begin{array}{cc}
	P_{1:i-1} &
		\left[\begin{array}{ll}
		(1-\alpha_{i}) P_{1:i-1}^-\\
		(1-\gamma_{i} ) P_{1:i-1}^+
		\end{array} \right] \\
		0  & 		  \mu_{i} P_{1:i-1}
		\end{array} \right]  \\
	&=[ \pi_{1:i-1}, \pi_{1:i}^+ (I - \mu_{i} P_{1:i-1}) + \pi_{1:i}^+ \mu_{i}P_{1:i-1} ]\\
	&= \tilde{\pi}_{1:i}^T
	\end{align*}}Therefore, $\tilde{\pi}_{1:i}$ is the left eigenvector of $\tilde{P}_{1:i}$ corresponding to eigenvalue $1$, and as a result $\pi_{1:i}= \tilde{\pi}_{1:i} T^{-1} = [ \pi_{1:i-1} - \pi_{1:i}^+,\pi_{1:i}^+ ]$ is the stationary distribution under $P_{1:i}$. 
	
	Notice that $b_i$, the probability of $s_i=1$ under the stationary distribution, is $b_i = \pi_{1:i}^+ \mathbf{1}$ (where $\mathbf{1}$ is the all one column vector of appropriate dimension), the summation of the right half of $\pi_{1:i}$ per the way that the states are ordered. Therefore,
	{\small
	\begin{align}
	b_{i}& = \pi_{1:i}^+ \mathbf{1} \nonumber \\
	&\stackrel{(a)}{=} \pi_{1:i-1}  \left[\begin{array}{ll}
	(1-\alpha_{i}) P_{1:i-1}^-\\
	(1-\gamma_{i} ) P_{1:i-1}^+
	\end{array} \right]  (I - \mu_{i}P_{1:i-1})^{-1} \mathbf{1} \nonumber\\
	&\stackrel{(b)}{=}\frac{1}{1-\mu_{i}} \pi_{1:i-1}  \left[\begin{array}{ll}
	(1-\alpha_{i}) P_{1:i-1}^-\\
	(1-\gamma_{i} ) P_{1:i-1}^+
	\end{array} \right]  \mathbf{1}\nonumber\\
	&\stackrel{(c)}{=} \frac{1}{1-\mu_{i}} \pi_{1:i-1}  \left[\begin{array}{ll}
	(1-\alpha_{i}) \mathbf{1}\\
	(1-\gamma_{i} ) \mathbf{1}
	\end{array} \right]  \nonumber\\
	&\stackrel{(d)}{=} \frac{1-\alpha_{i}}{1-\mu_{i}} (1-b_{i-1}) + \frac{1-\gamma_{i}}{1-\mu_{i}}  b_{i-1} \nonumber\\
	&=  \frac{1-\alpha_{i}}{1-\mu_{i}} + \frac{\alpha_{i}-\gamma_{i}}{1-\mu_{i}}  b_{i-1}\label{eq:proof_fast_decaying:b_recursive}
	\end{align}}where (a) is due to the definition of $\pi_{1:i}^+$, (c) is due to that each row of $P_{1:i-1}^-$ and $P_{1:i-1}^+$ sum up to $1$; (d) is due to the definition of $b_{i-1}$. For $(b)$, it is due to, since $\mu_i\in(-1,1)$,
	$$(I - \mu_{i}P_{1:i-1})^{-1} \mathbf{1} =\mathbf{1} + \sum_{\tau=1}^\infty \mu_i^\tau P_{1:i-1}^\tau \mathbf{1} = \frac{1}{1-\mu_i}\mathbf{1}. $$ 
	We can expand \eqref{eq:proof_fast_decaying:b_recursive} recursively and get the desired result.
		\begin{align*}
	    b_i = b_{i-k}\prod_{j=i-k+1}^i \frac{\alpha_j - \gamma_j}{1 - \mu_{j}} + \sum_{j=i-k+1}^i \frac{1-\alpha_j}{1-\mu_j} \prod_{\ell=j+1}^i \frac{\alpha_\ell - \gamma_\ell}{1-\mu_\ell} 
	\end{align*}
	
\ifthenelse{\boolean{isfv}}{\input{appendix_supplementary.tex}}{}

\end{document}

%% file: appendix_supplementary.tex
\subsection{Proof of Proposition~\ref{prop:ergodic}}
It suffices to prove the line case, where the root node is $1$ and node $i$'s parent is node $i-1$. We fix the policy $\zeta = (\zeta_1,\ldots,\zeta_n)$ and from now on drop the dependenc on $\zeta$. The same as Appendix-\ref{subsec:appendix_proof_sensitivity}, we parameter the local transition (under policy $\zeta$) as
{\small
\begin{subeqnarray}\label{eq:p_j_parameterization}
P_j(s_j(t+1)  | s_j(t) , s_{j-1}(t) =0) &=& \left[\begin{array}{ll}
\alpha_j & 1-\alpha_j\\
\beta_j & 1 - \beta_j
\end{array}\right]\\
P_j(s_j(t+1)  | s_j(t) , s_{j-1}(t) =1) &=& \left[\begin{array}{ll}
\gamma_j & 1-\gamma_j\\
\omega_j & 1 - \omega_j
\end{array}\right] \\
P_1(s_1(t+1)  | s_1(t)  ) &= &\left[\begin{array}{ll}
\alpha_1 & 1-\alpha_1\\
\beta_1 & 1 - \beta_1
\end{array}\right] 
\end{subeqnarray}}where we also introduce $\gamma_1 = \alpha_1$ and $\omega_1 = \beta_1$ for notational consistency. The condition of Proposition~\ref{prop:ergodic} means that for all $j$, transition matrix 
$$ \left[\begin{array}{ll}
\alpha_j & 1-\alpha_j\\
\beta_j & 1 - \beta_j
\end{array}\right],\left[\begin{array}{ll}
\gamma_j & 1-\gamma_j\\
\omega_j & 1 - \omega_j
\end{array}\right]  $$
is ergodic, meaning $|\alpha_j - \beta_j|<1$, $|\gamma_j - \omega_j|<1$. 

Notice that for $1\leq i \leq n$, the state $(s_1,\ldots,s_i)$ form a self complete Markov chain, with transition matrix denoted as $P_{1:i}$. The rows and columns of matrix $P_{1:i}\in\R^{2^i\times 2^i}$, are indexed by state tuple $(s_1,s_2,\ldots,s_i)\in\{0,1\}^i$. We use the following ordering that maps each state tuple $(s_1,s_2,\ldots,s_i)$ to a row (column) number in $\{1,\ldots,2^i\}$.
$$\{0,1\}^i\ni  (s_1,s_2,\ldots,s_i) \mapsto 1+\sum_{j=1}^i 2^{j-1} s_j \in \{1,\ldots,2^i\}$$

We claim that for all $1\leq i\leq n$, the transition matrix $P_{1:i}$ is ergodic. We show this for induction. It is true for $i=1$. Assume it is true for $i-1$. We divide $P_{1:i-1}$ into the upper half and lower half as in the following.
$$P_{1:i-1} = \left[\begin{array}{ll}
P_{1:i-1}^-\\
P_{1:i-1}^+
\end{array} \right] $$ 
Per the ways the rows and columns are ordered, it is easy to check,
{\small$$P_{1:i} = \left[\begin{array}{ll}
\alpha_{i} P_{1:i-1}^- & (1-\alpha_{i})  P_{1:i-1}^- \\
\gamma_{i}  P_{1:i-1}^+ & (1-\gamma_{i})  P_{1:i-1}^+\\
\beta_{i}  P_{1:i-1}^- & (1-\beta_{i})  P_{1:i-1}^-\\
\omega_{i}  P_{1:i-1}^+ & (1-\omega_{i})  P_{1:i-1}^+
\end{array} \right].$$}We do the following similarity transform.
{\small
\begin{align*}
P_{1:i} &=	\overbrace{ \left[\begin{array}{cc}
I & 0\\
I & I
\end{array}\right] }^{:=T} {\left[ \begin{array}{cc}
		    \left[\begin{array}{ll}
			P_{1:i-1}^-\\
			P_{1:i-1}^+
		\end{array} \right] &
		 \left[\begin{array}{ll}
			(1-\alpha_{i}) P_{1:i-1}^-\\
			(1-\gamma_{i} ) P_{1:i-1}^+
		\end{array} \right] \\
		0  & 		  \left[\begin{array}{ll}
			(\alpha_{i} - \beta_{i}) P_{1:i-1}^-\\
			(\gamma_{i} -\omega_{i}) P_{1:i-1}^+
		\end{array} \right] 
	\end{array} \right] }\\
	&\quad \times  \left[\begin{array}{cc}
I & 0\\
-I & I
	\end{array}\right]\\
	&=	T \overbrace{\left[ \begin{array}{cc}
	P_{1:i-1} &
		\left[\begin{array}{ll}
		(1-\alpha_{i}) P_{1:i-1}^-\\
		(1-\gamma_{i} ) P_{1:i-1}^+
		\end{array} \right] \\
		0  & 		 \left[\begin{array}{ll}
			(\alpha_{i} - \beta_{i}) P_{1:i-1}^-\\
			(\gamma_{i} -\omega_{i}) P_{1:i-1}^+
		\end{array} \right]  
		\end{array} \right] }^{:=\tilde{P}_{1:i}}  T^{-1} = T \tilde{P}_{1:i} T^{-1}
\end{align*}}
Now $P_{1:i}$ share the spectrum with $\tilde{P}_{1:i}$. Matrix $\tilde{P}_{1:i}$ is upper block-triangular. The upper left diagonal block is $P_{1:i-1}$, which by the induction assumption has a unique eigenvalue $1$ and all other eigenvalues are strictly smaller than $1$ (in magnutude). For the lower right block, 
$$ \left[\begin{array}{ll}
			(\alpha_{i} - \beta_{i}) P_{1:i-1}^-\\
			(\gamma_{i} -\omega_{i}) P_{1:i-1}^+
		\end{array} \right]  $$
each of its row have the same sign and sum up to either $(\alpha_{i} - \beta_{i})$ or $(\gamma_{i} -\omega_{i})$, both of which are strictly less than $1$ in absolute value. Therefore, $\tilde{P}_{1:i}$ has a unique eigenvalue $1$, and all other eigenvalues are strictly less than $1$ in absolute value. As a result, $P_{1:i}$ is ergodic, and the induction is finished.

Therefore, $P_{1:n}$ is ergodic, concluding the proof.